\documentclass[a4paper,10pt]{article}
\usepackage[top = 1in, bottom = 1in, left=0.8in, right=0.8in]{geometry}

\usepackage{graphicx} % Required for inserting images
\usepackage[utf8]{inputenc}
\usepackage{amsfonts}
\usepackage{hyperref}

\usepackage{amsmath}
\usepackage{amsthm}
\usepackage{tikz}
\usetikzlibrary{arrows.meta}
\usepackage{enumerate}
\usepackage{caption}
\usepackage{adjustbox}
\usepackage{appendix}
\usepackage{algpseudocode}

\newtheorem{thm}{Theorem}[section]
\newtheorem{cor}[thm]{Corollary}
\newtheorem{prop}[thm]{Proposition}
\newtheorem{lem}[thm]{Lemma}
\newtheorem{quest}[thm]{Question}
\newtheorem{claim}[thm]{Claim}

\theoremstyle{definition}
\newtheorem{defn}[thm]{Definition}
\newtheorem{obs}[thm]{Observation}

\theoremstyle{remark}
\newtheorem{rem}[thm]{Remark}

\makeatletter
\let\c@equation\c@thm
\makeatother

\newcommand{\startvx}{a}
\newcommand{\evx}{b}
\newcommand{\firstskipvx}{x}
\newcommand{\secondskipvx}{y}
\newcommand{\firsthelper}{v}
\newcommand{\secondhelper}{w}
\newcommand{\menge}[1]{C_{#1}}
\newcommand{\mengeeingeschrankt}[1]{A_{#1}}
\newcommand{\ivfirst}[1] {\mathcal{C}(#1)}
\newcommand{\ivsec}[1] {\mathcal{A}(#1)}
\newcommand{\configuration}{c}
\newcommand{\type}[2]{t_{#1}(#2)}
\newcommand{\shorttype}[1]{t_{#1}}
\newcommand{\restriction}[1]{\phi_#1}

\newcommand{\engquerkante}{lifting edge}
\newcommand{\cube}[1]{\mathcal{Q}_{#1}^3}
\newcommand{\kuniformcube}[1]{\mathcal{Q}_{#1}^k}

\newcommand{\abr}[1]{\langle #1 \rangle}

\newcommand{\lpskipped}{almost Hamilton path}
\newcommand{\allnc}{\texttt{createNormalisedConfigs}}
\newcommand{\setallnc}{N}
\newcommand{\coveredc}{\texttt{createWitnessedConfigs}}
\newcommand{\notcoveredc}{\texttt{createMissingConfigs}}
\newcommand{\setallcoveredc}{L}
\newcommand{\setallnotcoveredc}{M}

\usepackage{tikz}
\usetikzlibrary{decorations.pathreplacing}
\usetikzlibrary{math}
\usetikzlibrary{backgrounds}
\usetikzlibrary{calc}

\title{Loose Hamilton paths in the 3-uniform cube hypergraph}

\author{Oliver Cooley, Johannes Machata, Matija Pasch \footnote{\texttt{cooley@math.lmu.de, j.machata@campus.lmu.de, matijapasch@gmx.de}}}

\date{}

\begin{document}

\maketitle

\abstract{It is well-known that the $d$-dimensional hypercube contains a Hamilton cycle for $d\ge 2$. In this paper we address the analogous problem in the $3$-uniform cube hypergraph, a $3$-uniform analogue of the hypercube: for simple parity reasons, the $3$-uniform cube hypergraph can never admit a loose Hamilton cycle in any dimension, so we do the next best thing and consider loose Hamilton paths, and determine for which dimensions these exist.}

\section{Introduction}\label{sec: introduction}
\subsection{Motivation: The graph case}

The hypercube $Q_d$ of dimension $d$ is one of the most intensively studied objects in graph theory; it is defined as the graph on vertex set $\{0,1\}^d$, i.e.\ whose vertices are $01$-sequences of length $d$, in which two vertices are adjacent if the corresponding sequences differ in precisely one coordinate. It is a classic exercise in elementary graph theory to show that, for any integer $d\ge 2$, this graph contains a Hamilton cycle, i.e.\ a cycle containing all of the vertices.
Indeed, many far stronger results have been shown, for example regarding the number of Hamilton cycles~\cite{Clark97,DG75,Douglas77,FS09,Mollard88}, Hamilton cycles containing fixed matchings~\cite{AAAHST15,Fink07,Gregor09} and the robustness of Hamiltonicity against edge-deletion~\cite{CL91,CEGKO21}. 

Perhaps the easiest proof that the hypercube contains a Hamilton cycle is by induction on $d$, and utilises the fact that $Q_d$ can be constructed (up to isomorphism) from two copies of $Q_{d-1}$
by adding an edge between each pair of corresponding vertices from each copy.
Now for $d=2$ the graph is simply a cycle of length $4$, while for larger $d$ we can take identical Hamilton cycles in each copy of $Q_{d-1}$, delete two corresponding edges, one from each cycle, and add the matching edges between the endpoints of the deleted edges.

The aim of this paper is to examine possible extensions of this result to $3$-uniform hypergraphs.
\footnote{An extended version of this work appears in~\cite{JMthesis}: this contains some additional figures and more detailed explanations as well as some further results which do not appear in this paper, but does not include the pseudo-code in Appendix~\ref{app:code}.}

\subsection{Hypergraphs}

Given an integer $k\ge 2$, a \emph{$k$-uniform hypergraph} is a pair $(V,E)$, where $E \subseteq \binom{V}{k}$ is a set of unordered $k$-tuples of elements of $V$. We refer to the elements of $V$ as \emph{vertices} and the elements of $E$ as \emph{edges}. Thus a $2$-uniform hypergraph is simply a graph. In this paper, we focus on $3$-uniform hypergraphs, which for brevity we will simply refer to as hypergraphs.

There is a natural hypergraph analogue of the hypercube of dimension~$d$.

\begin{defn}\label{def: 3-uniform hypercube}
    Given $d\in \mathbb{N}_0$, the \emph{cube hypergraph of dimension $d$}, denoted $\cube{d}$, is the hypergraph on vertex set $V_d:=\{0,1,2\}^d$ whose edges are all unordered triples of (pairwise distinct) sequences $\{a,b,c\}$ which are identical on $d-1$ coordinates.
\end{defn}
Note that the possible values $0,1,2$ for the final remaining coordinate must each be taken by one of $a,b,c$.

\begin{figure}[h]
    \centering
    \includegraphics{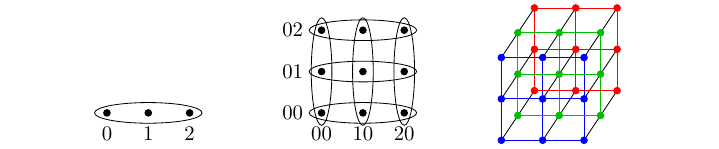}
    \captionsetup{width=0.88\textwidth, font=small}
    \caption{The cube hypergraphs $\cube{1},\cube{2}$ and $\cube{3}$. Note that in dimension 3, for clarity we use straight lines rather than ellipses to represent edges, and have omitted the vertex labels. Nevertheless, despite the optics, this represents a hypergraph rather than a graph; there is no ``middle'' or ``end'' vertex in an edge.}
    \label{fig:Cubesdim123}
\end{figure}

Of course, one can naturally extend this definition further to \emph{$k$-uniform cube hypergraphs}, which were first introduced by Burosch and Ceccherini, who studied isomorphic subgraphs~\cite{BC95} and characterisations~\cite{BC96} of cube hypergraphs, while Dvo\v{r}\'{a}k and Valla studied the automorphism groups of cube hypergraphs~\cite{DV21}.
Furthermore, Wang, Meng and Tian studied the vertex- and edge-connectivity of cube hypergraphs~\cite{WJT22}.
Otherwise, we are not aware of any further study of these hypergraphs and many very natural questions remain open. In particular, we focus on Hamilton paths and cycles.

In ($3$-uniform) hypergraphs there are several potential ways to generalise the notion of a path or cycle. In this paper we consider loose paths.

\begin{defn}\label{def: loosepath}
    Given $\ell \in \mathbb{N}_0$,
    a \emph{loose path of length $\ell$} in a hypergraph consists of a sequence of distinct vertices $v_0,v_1,\ldots{},v_{2\ell}$
    and a sequence of distinct edges $e_1,e_2,\ldots ,e_\ell$
    such that $e_i = \{v_{2i-2},v_{2i-1},v_{2i}\}$ for $i\in \{1,\ldots,\ell\}$.

    A \emph{loose cycle} of length $\ell$ is defined identically except that $v_0=v_{2\ell}$, and we demand that $\ell \neq 0$.
\end{defn}

We refer to $v_0$ and $v_{2\ell}$ as the \emph{endpoints} of the loose path, and speak of a loose path \emph{from $v_0$ to $v_{2\ell}$}.

Note that a loose path intrinsically comes with an order of its vertices, and therefore with a direction. However, with a slight abuse of terminology we often identify a path with its edge set, in which case a path and its reverse orientation become identical. Furthermore, with this identification, switching the order of the first two vertices or the last two vertices does not change the path -- thus precisely the same set of edges also gives rise to a loose path between either of $\{v_0,v_1\}$ and either of $\{v_{2\ell-1},v_{2\ell}\}$.

We briefly discuss other potential generalisations of paths and cycles in Section~\ref{sec: concluding remarks}.

\subsection{Main result}
It is now natural to ask whether $\cube{d}$ contains a \emph{loose Hamilton cycle}, i.e.\ a loose cycle containing all vertices. A moment's thought shows that this can never occur: a loose cycle necessarily contains an even number of vertices, while $\cube{d}$ contains $3^d$ vertices, which is always an odd number.

Thus the natural generalisation of the graph result mentioned above is obviously false.
On the other hand, a loose path contains an odd number of vertices, so we may weaken the question to ask whether a \emph{loose Hamilton path} exists, i.e.\ a loose path containing all vertices.
We even consider the following stronger property: we say that a hypergraph is \emph{loose Hamilton connected} if for any two distinct vertices $a,b$ there exists a loose Hamilton path from $a $ to $b$.
Our main result is the following.
 
\begin{thm}\label{thm:lHc}
The cube hypergraph $\cube{d}$ is loose Hamilton connected for $d\ge 4$.
\end{thm}

It is a trivial check that $\cube{d}$ is loose Hamilton connected for $d=1$ (and even degenerately for $d=0$), but does not even contain \emph{any} loose Hamilton paths for $d=2$. A slightly harder, but still elementary exercise is to show that for $d=3$ the hypergraph $\cube{d}$ does not contain a loose Hamilton path -- a proof of this fact can be found in~\cite{JMthesis}. Thus Theorem~\ref{thm:lHc} covers all of the difficult cases.

\subsection{Proof outline}\label{subsec: proof outline}

 The first natural approach to proving Theorem~\ref{thm:lHc} is an analogous proof to the graph case:
 split the hypergraph $\cube{d}$ into three ``layers'', each isomorphic to $\cube{d-1}$, find a loose Hamilton path in each layer, and then connect the loose paths in each layer using additional edges.
 However, the connecting edges used to jump between two layers also contain a vertex in the third layer, which must then be avoided by the path in that layer.
 
 To take account of this, we consider a stronger induction hypothesis:
 for every four distinct vertices $\startvx, \evx, \firstskipvx, \secondskipvx$ in $V_d$ there is a loose path from $\startvx$ to $\evx$ omitting $\firstskipvx \text{ and } \secondskipvx$ and covering all other vertices. We call such an ordered $4$-tuple of vertices a \textit{configuration} and we call a configuration \textit{covered} if there exists a loose path satisfying these conditions. One might hope to prove that for $d \geq 4$, we can cover every configuration. Unfortunately, this is false for some smaller dimensions, in particular for $d=4$.

 Therefore,
in Section $\ref{sec: induction hypothesis}$ we introduce the set $\menge{d}$, which consists of five types of configurations in $\cube{d}$. We call these types $\type{1}{d}, \type{2}{d},\type{3}{d}, \type{4}{d}\text{ and } \type{5}{d}$ and introduce the hypothesis:\\
\indent \parbox{14cm}{
\begin{itemize}
    \item[$\ivfirst{d}$:]
For every configuration $(\startvx,\evx,\firstskipvx,\secondskipvx) \in \menge{d}$, there exists a loose path in $\cube{d}$ from $\startvx$ to $\evx$ omitting $\firstskipvx$ and $\secondskipvx$ and covering every other vertex in $V_d$. 
\end{itemize}
}\\
Introducing this condition allows us to use the omitted vertices $x,y$ to take account of the fact that some vertices within a layer are already used by edges which connect layers. 

One might now hope to prove inductively that $\ivfirst{d}$ holds for all $d \geq 4$.
Unfortunately, even this turns out to be false (though true for $d \geq 5$.). Specifically, there are precisely four uncovered configurations in $\menge{4}$ (up to symmetries). This means the base case of our induction would have to be $d=5$. Unfortunately, a computer search for $d=5$ proved to be unmanageable in testing every configuration in $\menge{5}$. 

To circumvent this problem, we modify our induction such that it also holds for $d=4$, which is more tractable.\footnote{An alternative approach is to consider $\menge{5}$ manually,
and find paths to cover all the configurations. This is indeed manageable, but slightly more involved than the approach we take in this paper.
}
More precisely, in Section~$\ref{sec: induction hypothesis}$ we introduce the types $\restriction{1}, \restriction{2}, \restriction{3}$ and $\restriction{4}$ to further restrict the set $\menge{d}$ -- the new restricted set we call $\mengeeingeschrankt{d}$. In particular, $\mengeeingeschrankt{4}$ does not contain any of the 
uncovered configurations in $\menge{4}$. We now consider the following property:\\
\indent \parbox{14cm}{
\begin{itemize}
    \item[$\ivsec{d}$:]
For every configuration $(\startvx,\evx,\firstskipvx,\secondskipvx) \in \mengeeingeschrankt{d}$, there exists a loose path from $\startvx$ to $\evx$ omitting $\firstskipvx$ and $\secondskipvx$ and covering every other vertex in $V_d$.
\end{itemize}
}\\
In Section~$\ref{sec:configtypes}$ we will formally define $\menge{d}$ and $\mengeeingeschrankt{d}$. From these definitions, the following observation is immediate.
\begin{obs}\label{observationteilmenge}
    $\mengeeingeschrankt{d} \subseteq \menge{d} \ for\ all\ d \geq 4 $.
\end{obs}

The following is an immediate corollary of Observation~\ref{observationteilmenge} and the definitions of the events $\ivfirst{d}$ and $\ivsec{d}$.

\begin{cor} \label{CausA}
    $\ivfirst{d} \implies \ivsec{d}$ for all $ d \geq 5$.\footnote{Note that formally this statement would be true even without the condition $d\geq 5$. We include this condition because it turns out that $\ivfirst{d}$ is only true, and therefore the implication is only a non-empty statement, if $d\geq 5$.
}\qed
\end{cor}

Next we address the base case with the following proposition.
\begin{prop}\label{dim4}
    $\ivsec{4}$ holds.
\end{prop}
Proposition~$\ref{dim4}$ is proved using computer assistance, see Section~\ref{sec: base case}.
We remark, though, that it is actually easy to verify directly that $\cube{4}$ is loose Hamilton connected (a weaker condition than $\ivsec{4}$) --
the interested reader can find a proof in Appendix~A of~\cite{JMthesis}.
Since intuitively higher dimensions allow more freedom to construct loose paths,
this makes the general statement of Theorem~\ref{thm:lHc} for $d \geq 4$ at least plausible.

Next we proceed to the induction step. This can be broken down into Corollary~\ref{CausA} and the following two results.
\begin{lem} \label{d-1aufd}
	$\ivsec{d-1} \implies \ivfirst{d}$ for all $d\geq 5$.
\end{lem}

\begin{prop} \label{cube is lHC}
     $\ivsec{d} \implies \cube{d}$ is loose Hamilton connected for all $ d \geq 4$.
\end{prop}

The structure of the proof and the relevant implications are summarised in the following figure: \\
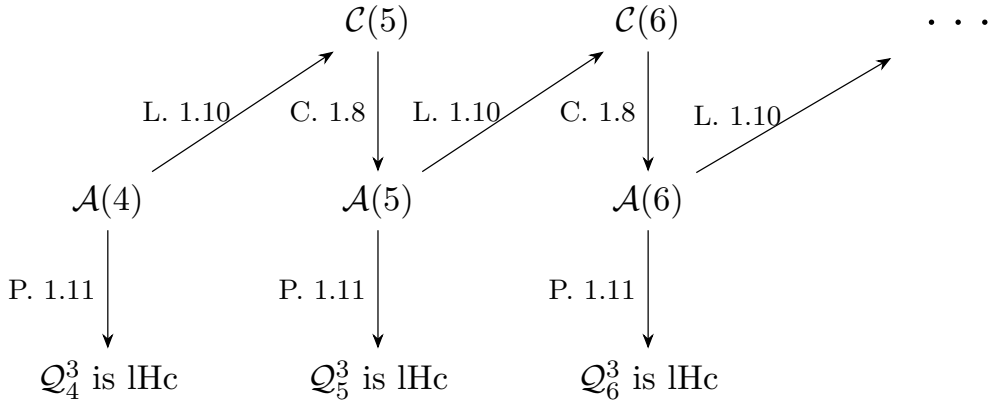
\begin{figure}[h]
    \centering\resizebox{0.8\textwidth}{!}{
\begin{tikzpicture}[>=Stealth]
    % Nodes
    \node (A1) at (0,0) {$\ivsec{4}$};
    \node (A2) at (3,0) {$\ivsec{5}$};
    \node (A3) at (6,0) {$\ivsec{6}$};
    \node (C2) at (3,2) {$\ivfirst{5}$};
    \node (C3) at (6,2) {$\ivfirst{6}$};
    \node (C4) at (9.5,2) {\huge \ldots{} };
    \node (D1) at (0,-2) {$\cube{4}$ is lHc};
    \node (D2) at (3,-2) {$\cube{5}$ is lHc};
    \node (D3) at (6,-2) {$\cube{6}$ is lHc};
    
    % Arrows
    \draw[->] (A1) -- (C2) node[midway, left] {\footnotesize L. \ref{d-1aufd}};
    \draw[->] (C2) -- (A2) node[midway, left] {\footnotesize C. \ref{CausA}};
    \draw[->] (A2) -- (C3) node[midway, left] {\footnotesize L. \ref{d-1aufd}};
    \draw[->] (C3) -- (A3) node[midway, left] {\footnotesize C. \ref{CausA}};
    \draw[->] (A1) -- (D1) node[midway, left] {\footnotesize P. \ref{cube is lHC}};
    \draw[->] (A2) -- (D2) node[midway, left] {\footnotesize P. \ref{cube is lHC}};
    \draw[->] (A3) -- (D3) node[midway, left] {\footnotesize P. \ref{cube is lHC}};
    \draw[->] (A3) -- (8.7,1.6) node[midway, left] {\footnotesize L. \ref{d-1aufd}};
  
\end{tikzpicture}}

  \captionsetup{width=0.88\textwidth, font=small}
  \caption{Our induction step ``$\ivsec{d-1} \implies \ivsec{d}$'' follows from Lemma~\ref{d-1aufd} and Corollary~\ref{CausA}. 
Proposition~\ref{cube is lHC} creates the connection between our hypotheses and Theorem~\ref{thm:lHc}, the loose Hamilton connectedness of the cube hypergraph $\cube{d}$ for all $d \geq 4$. Note that ``lHc'' is an abbreviation for ``loose Hamilton connected''.
}
    \label{fig:induction structure}
\end{figure}

Corollary~\ref{CausA} follows immediately from the definitions and Proposition~\ref{cube is lHC} can be deduced through elementary arguments. Thus, the main difficulty in the inductive step is to prove Lemma~\ref{d-1aufd}.

\section{Preliminaries}\label{sec: preliminaries}

\subsection{Definitions and Notation}

For an integer $n \in \mathbb{N}$, we denote by $[n]:= \{1,2,\ldots,n\}$ the set of the natural numbers up to $n$.

Formally, a vertex $a \in V_d$ is a sequence $(a_1,a_2,\ldots,a_d)$, where $a_i \in \{0,1,2\}$.
We will often denote such a sequence by $a_1a_2\ldots a_d$.

Let $v \in V_{d-1}$ and $ i \in \{0,1,2\}$. With a slight abuse of notation we use $(v,i)$ to denote the sequence in $V_d$ obtained by appending $i$ to $v$, for example ${(1020,1)} = {10201}$. We will occasionally also denote this by $vi$ for simplicity.

In the other direction, suppose $v \in V_{d}$ and $\tilde{d} \in \{1,\ldots{},d-1\}$.
We denote by $v_{[\tilde{d}]}$ the restriction of $v$ to its first $\tilde d$ coordinates (i.e.\ we truncate the sequence after $\tilde d$ entries). This will most frequently be used with $\tilde d = d-1$, but also occasionally with $\tilde d = d-2$.

The following observation is trivial, but essential to our proof strategy.
\begin{rem} For $d \in \mathbb{N}_0$ the cube hypergraph $\cube{d}$ can be recursively constructed as follows:
\begin{itemize}
    \item $\cube{0}$ consists of a single vertex (the empty sequence) and no edges.
\item For $d\ge 1$, to construct $\cube{d}$, take three copies of $\cube{d-1}$, which we call \textit{layers} $L_0,L_1$ and $L_2$: within layer $L_i$, each vertex $v\in V_{d-1}$ has a copy $vi$, and for every edge $\{u,v,w\}$ of $\cube{d-1}$, we have the edge $\{ui,vi,wi\}$.
In addition for every vertex $v \in V_{d-1}$, we add the edge $\{v0,v1,v2\}$; such edges are called \textit{\engquerkante}s. 
\end{itemize}
\end{rem}

We say that we \textit{split at coordinate i} to mean that we swap coordinate~$i$ with the last coordinate (see Section~\ref{subsec: symmetries} on symmetries) so that it now determines the layers.

We will often use the following abbreviations for frequently occurring terminology.\\
\begin{center}
    \begin{tabular}{rl}
     lHc:& loose Hamilton connected \\
     LP:& loose path\\
     LHP: & loose Hamilton path
\end{tabular}
\end{center}

We use the standard notion of \textit{Hamming distance} between two vertices, defined as the number of coordinates in which the two vertices differ.

A \emph{configuration} in $\cube{d}$ (informally introduced in Section~\ref{subsec: proof outline}) is a tuple of distinct vertices in $V_d$. A configuration can have any number of vertices, but we will consider configurations of two or four vertices.

\subsection{Symmetries and normalisation}\label{subsec: symmetries}\label{sec:normalisation}
We will make heavy use of automorphisms: any permutation of the coordinates (from $[d]$) corresponds to an automorphism of $\cube{d}$, and similarly any permutation of the values (from $\{0,1,2\}$) in any one coordinate gives an automorphism. Naturally, we can also permute the values on any number of coordinates, and also combine the two possibilities.

For the rest of the paper, we refer to automorphisms as \emph{symmetries}, and will use them to drastically reduced the number of cases we have to consider. Importantly, a loose path under the action of a symmetry remains a loose path.

To reduce the number of possible scenarios in a systematic way we also introduce normalisation, which yields equivalence classes of configurations. 
We will introduce normalisation for configurations of four vertices, but it also works analogously for configurations of two vertices. To illustrate the normalisation procedure, consider the configuration $\startvx,\evx,\firstskipvx,\secondskipvx$ in $V_4$ with\footnote{Formally, the vertices of the configuration should be sequences of length four, but it is convenient to express them as column vectors.}\\
\begin{center}  
$\startvx=\begin{pmatrix}2\\2\\0\\1\end{pmatrix},\qquad
\evx=\begin{pmatrix}0\\2\\1\\1\end{pmatrix},\qquad
\firstskipvx=\begin{pmatrix}2\\0\\1\\1\end{pmatrix},\qquad
\secondskipvx=\begin{pmatrix}1\\0\\2\\1\end{pmatrix}$.
\end{center}

We start by writing the configuration in a matrix with every column representing a vertex. \\
\begin{center}
$\left( \begin{array}{cccc}
2&0&2&1\\
2&2&0&0\\
0&1&1&2\\
1&1&1&1
\end{array}\right)$\\
\end{center}

We permute the entries in each row by \textit{first traversal}, meaning we permute the values within a row such that, considering only entries of the row in which a value appears for the first time, these values are in ascending order and without gaps. 
(Equivalently, we choose the permutation of $\{0,1,2\}$ which lexicographically minimises the row.)
For example, consider the first row $(\ 2\ 0\ 2\ 1\ )$, and we choose the permutation $2\mapsto 0 \mapsto 1 \mapsto 2$, and the row becomes $(\ 0\ 1\ 0\ 2\ )$.
By permuting the values in each row (independently) in this fashion we obtain the matrix 
$$\begin{pmatrix}0 & 1 & 0 & 2\\0 & 0 & 1 & 1\\0 & 1 & 1 & 2\\0 & 0 & 0 & 0\end{pmatrix}.$$
The next step is to rearrange the rows lexicographically (corresponding to a permutation of coordinates). 
In our example, this gives us the matrix 
$$C:=\begin{pmatrix}0 & 0 & 0 & 0\\0 & 0 & 1 & 1\\0 & 1 & 0 & 2\\0 & 1 & 1 & 2\end{pmatrix}.$$

If our configuration is of two vertices we are now finished. For a configuration of four vertices we observe that wlog we can swap the two omitted vertices $\firstskipvx$ and $\secondskipvx$,
which corresponds to switching the third and fourth column of $C$. We then once again permute the values in each row by first traversal and rearrange the rows lexicographically to obtain the matrix 

$$\tilde C:=\begin{pmatrix}0 & 0 & 0 & 0\\0 & 0 & 1 & 1\\0 & 1 & 2 & 0\\0 & 1 & 2 & 1\end{pmatrix}.$$

Finally, we have to choose which of the configurations $C$ and $\tilde C$ we will use, and
we choose the lexicographically smaller one (where we read the entries one row at a time).
In our example, the first entry that differs between the two matrices is in the third entry of the third row, and $C$ has the smaller entry here, therefore $C$ is our choice.\footnote{It is of course possible that $C$ and $\tilde C$ might be identical for some starting configurations.}

This process generalises in a completely obvious way to configurations in $\cube{d}$ for any $d \in \mathbb{N}$.
A configuration is called \emph{normalised} if it can be obtained from some configuration by applying this process.

 We observe that in every step of the normalisation process we used symmetries and therefore we always stayed in the same equivalence class of configurations. In the remainder of the paper, we will only consider normalised configurations.

    \section{Induction hypothesis}\label{sec: induction hypothesis}
    Recall from Section~\ref{subsec: proof outline} that we defined $\ivfirst{d}$ and $\ivsec{d}$ to be the properties that all configurations in $\menge{d}$ or $\mengeeingeschrankt{d}$ respectively can be covered by loose paths. In this section we formally define the sets $\menge{d}$ and $\mengeeingeschrankt{d}$.
    To do this we have to define the types $\type{1}{d},\type{2}{d},\type{3}{d}, \type{4}{d}, \type{5}{d}$ and $\restriction{1},\restriction{2},\restriction{3},\restriction{4}$, which were also mentioned in Section~\ref{subsec: proof outline}.

\subsection{Configuration types}\label{sec:configtypes}

    \begin{defn}
    We say that a configuration $(a,b,x,y)$ is \emph{of type $t_i(d)$} for $i \in [5]$ if the following holds (see Figure~\ref{fig: types}):

\vspace{0.2cm}

\hspace{0.5cm}
\begin{minipage}{0.8\textwidth}
\begin{enumerate}[$t_1(d)$:]
\item $\startvx_d = \firstskipvx_d = 0,\evx_d=1, \secondskipvx_d =2$ and $\startvx_{[d-1]} \notin \{\evx_{[d-1]},\firstskipvx_{[d-1]},\secondskipvx_{[d-1]}\}$.
\item $\startvx_d = \evx_d = 0, \firstskipvx_d = 1, \secondskipvx_d =2$ and $\startvx_{[d-1]} \notin \{\evx_{[d-1]},\firstskipvx_{[d-1]},\secondskipvx_{[d-1]}\}$.
\item $\startvx_d=\firstskipvx_d=0$ and $\evx_d=\secondskipvx_d=1$.
\item $\startvx_d=0, \evx_d=\firstskipvx_d=1, \secondskipvx_d=2$ and $\evx_{[d-1]} \notin \{\startvx_{[d-1]},\firstskipvx_{[d-1]},\secondskipvx_{[d-1]}\}$.
\item $\startvx_d=\evx_d=0, \firstskipvx_d = 1,\secondskipvx_d=2$ and $\evx_{[d-1]} \notin \{\startvx_{[d-1]},\firstskipvx_{[d-1]},\secondskipvx_{[d-1]}\}$.
\end{enumerate}
\end{minipage}

\vspace{0.3cm}
    
    More generally we say a configuration $\configuration$ is of type $\type{i}{d}$, for $i \in [5]$, if using symmetries, and potentially switching $\firstskipvx,\secondskipvx$, we can transform the configuration $\configuration$ into a configuration $\tilde{\configuration}$ which has the form described above.
    \end{defn}
\begin{figure}[h]
\centering
\resizebox{0.8\textwidth}{!}{
\includegraphics{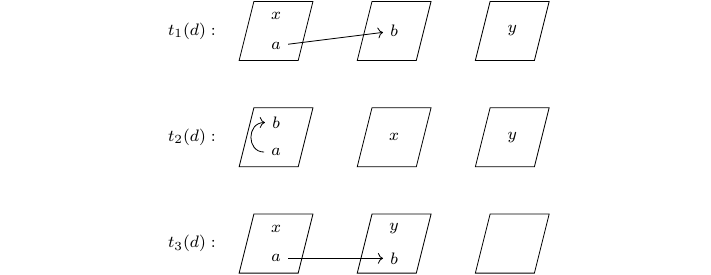}
}
\captionsetup{width=0.88\textwidth, font=small}
\caption{The types $\type{1}{d},\type{2}{d},\type{3}{d}$. As always, $\startvx$ represents the starting vertex, $\evx$ the end vertex and $\firstskipvx,\secondskipvx$ the two omitted vertices. For types $\type{1}{d}$ and $\type{2}{d}$, additionally the vertices $\evx,\firstskipvx$ and $\secondskipvx$ cannot lie on the same \text{\engquerkante} as $\startvx$. Types $\type{4}{d}$ and $\type{5}{d}$ can be obtained from types $\type{1}{d}$ and $\type{2}{d}$ respectively by switching $\startvx$ and $\evx$.}
\label{fig: types}
\end{figure}

Note that a configuration may be of more than one different type which mostly (but not solely) depends on which coordinate is chosen to split into layers.

    \begin{defn}
        The set $\menge{d}$ is the set of all normalised configurations of type $\type{1}{d},\type{2}{d},\type{3}{d},\type{4}{d}$ or $\type{5}{d}$.
    \end{defn}
    
Next, we define the types $\restriction{1},\restriction{2},\restriction{3}\text{ and }\restriction{4}$.

    \begin{defn} A configuration $\configuration$ is of type $\restriction{i}$ for $i\in [4]$ if the following holds.
\begin{enumerate}[$\phi_1$:]
    \item $\configuration$ is of type $\type{1}{d}$ split at the last coordinate and $\evx$ agrees in at least one coordinate with $\startvx, \firstskipvx$ or $\secondskipvx$ (formally, $\exists i \in [d-1]$ such that $\evx_i \in \{\startvx_i,\firstskipvx_i,\secondskipvx_i\}$).
    \item $\configuration$ is of type $\type{2}{d}$ or $\type{5}{d}$ split at the last coordinate and either $\startvx$ or $\evx$ agrees with either $\firstskipvx$ or $\secondskipvx$ in at least two coordinates (formally, $\exists I \in \binom{[d-1]}{2}$ and $\exists u \in \{\startvx, \evx\}, $ and $ v \in \{\firstskipvx, \secondskipvx\}$ such that $u_i=v_i$ for both $i \in I$).
    \item $\configuration$ is of type $\type{3}{d}$ split at the last coordinate and either $\startvx$ agrees in at least one coordinate with $\secondskipvx$ or $\evx$ agrees with in at least one coordinate with $\firstskipvx$ (formally, $\exists i \in [d-1]$ such that either $\startvx_i$ = $\secondskipvx_i$ or $\evx_i=\firstskipvx_i$).
    \item $\configuration$ is of type $\type{4}{d}$ split at the last coordinate and $\startvx$ agrees in at least one coordinate with $\evx, \firstskipvx$ or $\secondskipvx$ (formally, $\exists i \in [d-1]$ such that $\startvx_i \in \{\evx_i,\firstskipvx_i,\secondskipvx_i\}$).
\end{enumerate}

More generally we say a configuration $\configuration$ is of type $\restriction{i}$, for $i \in [4]$, if using symmetries,
and potentially switching $\firstskipvx,\secondskipvx$, 
we can transform the configuration $\configuration$ into a configuration $\tilde{\configuration}$ which has the form described above.

    \end{defn}

    \begin{defn}
        The set $\mengeeingeschrankt{d}$ is precisely the set of all normalised configurations of types $\restriction{1},\restriction{2},\restriction{3},\restriction{4}$.
    \end{defn}

Recall that Observation~\ref{observationteilmenge} states that $\mengeeingeschrankt{d}$ is a subset of $\menge{d}$ -- this follows directly from the definitions, since a configuration of type $\restriction{i}$ must also be of type $\type{i}{d}$ (or $\type{5}{d}$ if $i=2$).

\subsection{Induction base case}\label{sec: base case}
    Recall that Proposition~\ref{dim4} states that $\ivsec{4}$ holds, which represents the base case of our induction.
    \begin{proof}[Proof of Proposition~\ref{dim4}]
    We prove this proposition using computer assistance: we wrote and ran a computer program to check which normalised configurations can be covered. The output of the program shows that we can cover every normalised configuration except for the four configurations $A,B,C,D$ defined below; we must therefore check that these configurations are not in $\mengeeingeschrankt{4}$.
    Appendix~\ref{app:code} describes in more detail how to check that all other configurations
    are indeed covered.
    
    Specifically, the four uncovered configurations, in matrix form, are:\\\\

\[
\begin{adjustbox}{valign=c, margin=-1cm 0.3cm -1cm -1cm}
{\fontsize{8}{10}\selectfont 
     $A:= \begin{pmatrix}
0&1&1&2 \\
0&1&1&2 \\
0&1&2&1 \\
0&1&2&1 \end{pmatrix},
\quad
B:=  \begin{pmatrix}
0&1&0&2 \\ 
0&1&0&2 \\
0&1&2&0 \\
0&1&2&0 \end{pmatrix},
\quad
 C:= \begin{pmatrix}
0&0&1&2 \\ 
0&0&1&2 \\
0&0&1&2 \\
0&1&0&1 \end{pmatrix},
\quad
D:= \begin{pmatrix}
0&0&1&1 \\ 
0&0&1&1 \\
0&1&0&1 \\
0&1&0&1 \end{pmatrix}$}
\end{adjustbox}
\]
As in Section~\ref{sec:normalisation}, the vertices are represented by column vectors in the usual order $\startvx,\evx,\firstskipvx,\secondskipvx$.

We observe that $A$ is of type~$\type{4}{4}$ split at the first coordinate (with a switch of $\firstskipvx,\secondskipvx$). 
Observe that all other types $\type{i}{4}$ require $\startvx$ to agree with at least one other vertex in the splitting coordinate, but in $A$, the starting vertex $\startvx$ has Hamming distance~$4$ to all other vertices.
Therefore the configuration $A$ is only of type~$\type{4}{4}$. 
Analogously we see that $B$ is only of type~$\type{1}{4} $ and that $D$ is only of type~$\type{3}{4}$, while $C$ is of type~$\type{3}{4}$ if split at the last coordinate and of type~$\type{2}{4}$ and $\type{5}{4}$ if split at any other coordinate.

Let us check if $A$ is of any of the types $\restriction{1},\restriction{2},\restriction{3}$ or $\restriction{4}$. Since $A$ is only of type~$\type{4}{4}$ it can by definition only be of type~$\restriction{4}$. But $\startvx$ does not agree in any coordinate with any of the vertices $\evx, \firstskipvx, \secondskipvx$ and therefore  the configuration $A$ is by definition not of type~$\restriction{4}$.
Analogously, one can also check that $B,C,D$ are not of the types~$\restriction{1},\restriction{2},\restriction{3}$ or $\restriction{4}$.
We omit the details.

 Since all four configurations are not of types $\restriction{1},\restriction{2},\restriction{3}$ or $\restriction{4}$, they are by definition not in $\mengeeingeschrankt{4}$.
    \end{proof}

    \section{Induction step}\label{sec: induction step}

In this section, we will prove the induction step. Recall from Section~\ref{subsec: proof outline} that this is broken down into three steps, given by Corollary~\ref{CausA}, Lemma~\ref{d-1aufd} and Proposition~\ref{cube is lHC}. Corollary~\ref{CausA} follows directly from the definitions of $\ivsec{d}$ and $\ivfirst{d}$, so we now proceed with the remaining two steps.

   \begin{proof}[Proof of Proposition~\ref{cube is lHC}]
       Let $d \geq 4$ and let $\startvx,\evx$ be any two distinct vertices in $V_d$. We want to show that there is a loose Hamilton path between $\startvx \text{ and }\evx$.
       Wlog $\startvx \equiv 0$, meaning every coordinate of $\startvx$ is zero. We differentiate two cases.\\
        \textbf{Case 1: $\startvx$ and $\evx$ agree in at least one coordinate.}\\
        In other words, there is a coordinate $i \in \{1,\ldots,d\}$ with $\evx_i=0$. We split at this coordinate $i$ and observe that $\startvx$ and $\evx$ lie in the same layer. Let $u,v$ be the two vertices that share a \text{\engquerkante} with $\evx$.
        We convert $(\startvx,\evx)$ into a configuration of four vertices by choosing
        $(\startvx',\evx',\firstskipvx',\secondskipvx'):=(\startvx,u,\evx,v)$.
        For our new configuration $(\startvx',\evx',\firstskipvx',\secondskipvx')$ to be of type~$\restriction{1}$, it has to be of type~$\type{1}{d}$ and $\evx'$ has to agree with $\startvx',\firstskipvx'$ or $\secondskipvx'$ in at least one coordinate.
        Clearly our configuration is of type~$\type{1}{d}$ and $b'$ agrees in $d-1$ coordinates with $\firstskipvx'$ and $\secondskipvx'$. Therefore the configuration $(\startvx',\evx',\firstskipvx',\secondskipvx')$ is of type~$\restriction{1}$, so lies in $\mengeeingeschrankt{d}$ which means by assumption that it is covered by a loose path. By appending the \text{\engquerkante} $\{ u,v,\evx\}$ to this path, we obtain a loose Hamilton path starting at $\startvx$ and ending at $\evx$.
        
\begin{figure}[h]
\centering
\resizebox{1\textwidth}{!}{
\includegraphics{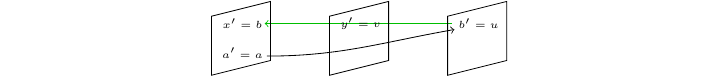}
}
\captionsetup{width=0.88\textwidth, font=small}
\caption{The black arrow represents the loose path from $\startvx$ to $u$ omitting the vertices $v$ and $\evx$. The green arrow represents the \text{\engquerkante} $(u,v,\evx)$.}
\label{fig: cube is lhc case 1}
\end{figure}
        \noindent \textbf{Case 2: $\startvx$ and $\evx$ differ in every coordinate.} Again, we convert $(\startvx,\evx)$ into a configuration of four vertices: let $u,v$ be the two vertices that share a \text{\engquerkante} with $\evx$ such that $v_d=\startvx_d$, and set $(\startvx',\evx',\firstskipvx',\secondskipvx'):=(\startvx,u,v,\evx)$. Analogously to case 1, we observe that our new configuration $(\startvx',\evx',\firstskipvx',\secondskipvx')$ is of type $\restriction{1}$ and therefore is covered by a loose path. Again, we append the \text{\engquerkante} $\{ u,v,\evx\}$ to obtain a loose Hamilton path starting at $\startvx$ and ending at $\evx$.

   \end{proof}

It remains to prove Lemma~\ref{d-1aufd}, which states that $\ivsec{d-1} \Rightarrow \ivfirst{d}$ for $d\geq 5$. The general strategy of the proof is displayed in Figure~\ref{fig_lpind1}: we can cover configurations of various types in $\cube{d}$ inductively, assuming the existence of appropriate loose paths within each layer.

 \begin{figure}[h]
 \centering
\resizebox{0.8\textwidth}{!}{
\includegraphics{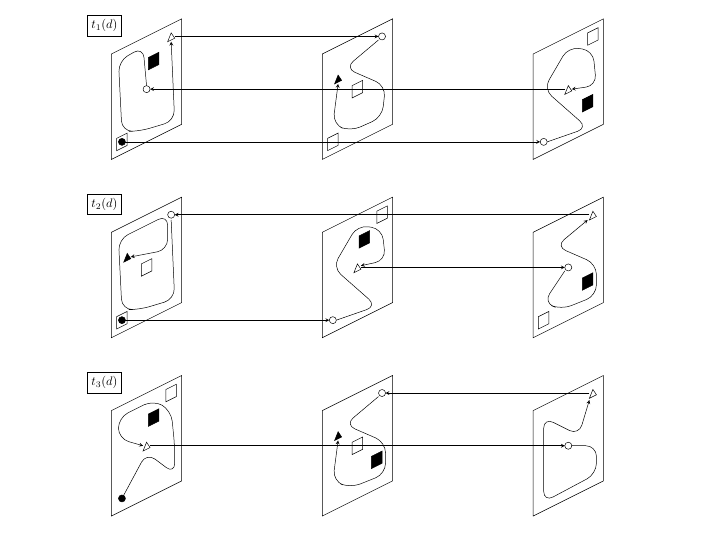}
}
\captionsetup{width=0.88\textwidth, font=small}
\caption{We can inductively construct configurations of types $\type{1}{d}, \type{2}{d}$ and $\type{3}{d}$, assuming that the necessary loose paths exist within each layer.
The filled shapes represent vertices of our configuration in $V_d$, while unfilled shapes indicate
what occurs within a layer; circles represent starting vertices, triangles are ending vertices and squares denote omitted vertices. The unfilled squares are covered by \text{\engquerkante}s which are used to jump between the other two layers.}
\label{fig_lpind1}
\end{figure}

The following proposition would clearly follow immediately from Lemma~\ref{d-1aufd} and Corollary~\ref{CausA} (for large $d$) -- however, we state and prove it here as a step on the way to proving Lemma~\ref{d-1aufd}.
    
\begin{prop}\label{lem d-1aufd}
            $\ivfirst{d-1} \implies \ivfirst{d}$.
\end{prop}

        \begin{proof}
            We observe that $\type{4}{d}$ and $\type{5}{d}$ are the inversions of $\type{1}{d}$ and $\type{2}{d}$, i.e.\ with the roles of $\startvx,\evx$ switched. Therefore it is sufficient to show that we can cover every configuration of the types $\type{1}{d},\type{2}{d}$ and $\type{3}{d}$. We aim to construct the necessary loose paths using loose paths that cover configurations from $\menge{{d-1}}$. 
            We consider the last two coordinates of our configuration to check that the loose paths within layers are of this form.
            Note that the type of a configuration already determines the last row, if we reorder rows such that we are splitting at the last coordinate.
            After normalisation there are 14 possibilities for the penultimate row. To see this, observe that the first entry (i.e.\ $a_{d-1}$) is always zero, the second is either 0 or 1, and the third and fourth entry lie in $\{0,1,2\}$. This means we have two possibilities for the second entry and three possibilities for both the third and fourth entry which gives us a total of $2 \cdot 3 \cdot 3 =18$ possibilities. However, we also observe that 0002, 0020, 0022 and 0021 are not valid rows in a normalised configuration, since distinct entries must be in ascending order according to their first appearance and not omit any values. This leaves us with the following 14 possibilities for the penultimate row:\\
            
           $ \begin{array}{ccc}
    \hspace{4cm} r_1=0000& \hspace{3cm} & r_8= 0102 \\
    \hspace{4cm}r_2=0001&& r_9=0110 \\
    \hspace{4cm}r_3=0010&& r_{10}=0111\\
    \hspace{4cm}r_4=0011&& r_{11}=0112\\
    \hspace{4cm}r_5=0012&& r_{12}=0120\\
    \hspace{4cm}r_6=0100&& r_{13}=0121\\
    \hspace{4cm}r_7=0101&& r_{14}=0122\\
\end{array}\\$

For each possible type of configuration, we consider all possibilities for the penultimate row.

\noindent \textbf{Case 1: $\configuration$ is of type $\type{1}{d}$.}
We know that $\startvx$ and $\firstskipvx$ are two distinct vertices in the same layer, and therefore they differ in a coordinate $i \in \{1,\ldots{},d-1\}$. We may assume wlog that $i=d-1$.
This means $r_1,r_2,r_6,r_7$ and $r_8$ cannot be our penultimate row. In Figure~\ref{fig_t1} we construct our desired loose path for configurations of type $\type{1}{d}$ with the penultimate row being $r_3,r_4,r_5,r_9,r_{10},r_{11},r_{12},r_{13}$ or $r_{14}$. We use the notation $a\abr{s}$ to denote $(a_{[d-1]},s)$ for $s=0,1,2$, i.e.\ the vertex obtained from $a$ by replacing the last entry with $s$ (which is simply $a$ if $s=a_d$).

To help interpret this figure, let us describe the case when the penultimate row is $r_3$ in detail. Here we split into layers according to the last coordinate. Since our configuration is of type $t_1(d)$, we have $a,x$ in layer $0$, we have $y$ in layer $1$ and $b$ in layer $2$. We now subdivide each layer into three sublayers according to the penultimate coordinate.
Since the penultimate row is $0010$ in case $r_3$, for example the vertex $x$ (the third vertex of the configuration) lies in sublayer $1$ of layer $0$.

We now pick two auxiliary vertices $v,w \in V_{d-1}$ with the property that $v_{d-1} = 1$ and $w_{d-1}=2$, and also such that $v_{[d-2]},w_{[d-2]}$ are distinct and do not lie in $\{a_{[d-2]},b_{[d-2]},x_{[d-2]},y_{[d-2]}\}$, which is possible since $|V_{d-2}| = 3^{d-2} \ge 9$.
In Figure~\ref{fig_t1}, copies of $v$ and $w$ appear in each layer -- formally, these are actually the vertices $vi$ and $wi$ for $i=0,1,2$, but we omit the last coordinate for clarity.

Now in each layer $L_i$ we find a path $P_i$ between the two vertices joined by an arrow and omitting the other two vertices displayed in the figure.
The types of the configuration to be covered within each layer are displayed beside it in the figure -- being of this type guarantees that the configuration is in $\menge{d-1}$ and therefore by the induction hypothesis can be covered by a loose path.

We now define the \text{\engquerkante}s $e_\startvx:=(\startvx,\startvx\abr{1},\startvx\abr{2})$,
$e_v:=(v2,v1,v0)$ and $e_w:=(w0,w2,w1)$. The concatenation $e_\startvx P_2 e_v P_0 e_w P_1$ is a loose path from $a$ to $b$ avoiding only $x$ and $y$.

More generally, whatever the penultimate row, we will choose auxiliary vertices $v,w\in V_{d-1}$, find loose paths $P_0,P_1,P_2$ covering the analogous vertices to the case of $r_3$ and concatenate the paths with the appropriate edges as above. The only things that change between cases are the choices of $v,w$ and the types of the paths that we find within a layer. These are displayed in the figure below (which in particular fixes the values of $v_{d-1},w_{d-1}$). Importantly, the condition in the definition of type $\type{1}{d}$ that $a_{[d-1]} \notin \{b_{[d-1]},x_{[d-1]},y_{[d-1]}\}$ implies that the edge $e_\startvx=(\startvx,\startvx\abr{1},\startvx\abr{2})$ avoids the vertices $b,x,y$, and we may choose $v,w$ such that there are also no other such conflicts arising from any \text{\engquerkante}s.

\begin{figure}[h]
\centering
    \includegraphics{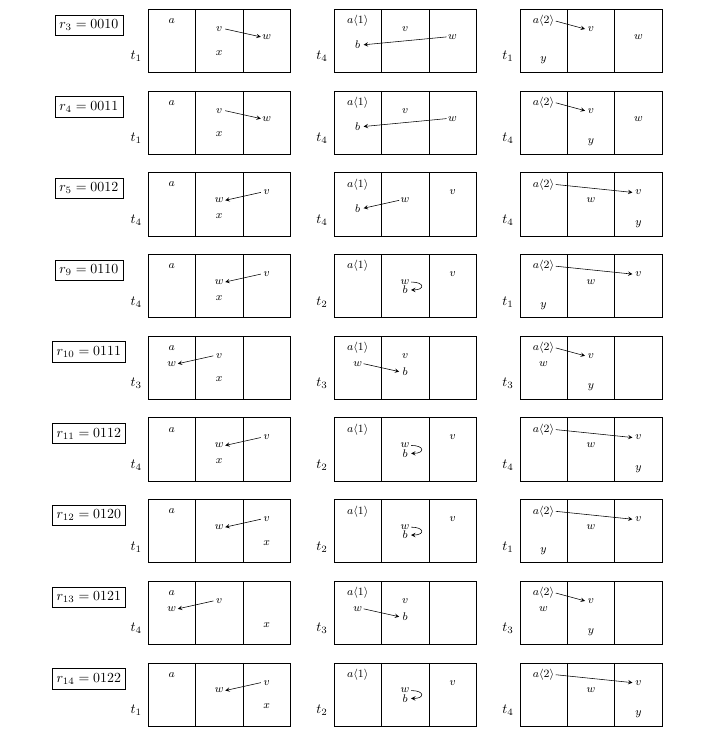}
\captionsetup{width=0.88\textwidth, font=small}
\caption{
The constructions for $\type{1}{d}$. 
Each layer is partitioned into sublayers according to coordinate $d-1$ and depicted from left (0) to right (2).
Each layer is labelled with the type of path used to cover (most of) the layer, where $\shorttype{i}$ should be understood as $\type{i}{d-1}$.
We use arrows to indicate the starting and ending vertex in each layer.
Moves between layers are not shown; these can be found in Figure \ref{fig_lpind1}.
}
\label{fig_t1}
\end{figure}

\noindent \textbf{Case 2: $\configuration$ is of type $\type{2}{d}$.}
Similarly to case~1 we split $\configuration$ at the last coordinate, and without loss of generality the final row is 0012.
We know that $\startvx$ and $\secondskipvx$ do not share a \text{\engquerkante}. This means there exists a coordinate $i \in [d-1]$ such that $\startvx_i \neq \secondskipvx_i$, and we assume wlog that $i=d-1$. In particular, this means the penultimate row cannot be $r_1,r_3,r_6,r_9$ or $r_{12}$.
We also observe that, since $\startvx_d=\evx_d$, the penultimate row $r_{11}$ can be obtained from $r_8$ by swapping $\startvx$ and $\evx$ and permuting the entry values.
Moreover, if the penultimate row is $r_8$, then by switching the last two coordinates we obtain a configuration of type $\type{1}{d}$ with penultimate row $r_5$, which was already covered in case~1.
This means we can exclude the rows $r_8$ and $r_{11}$.

Figure~\ref{fig_t2} shows how we construct loose paths $P_0,P_1,P_2$ within the appropriate layers for $r_2,r_4,r_5,$\linebreak$r_7,r_{10},r_{13}$ or $r_{14}$;
we concatenate the edge $(a,a\abr{2},a\abr{1})$, the path $P_1$, the edge $(v1,v0,v2)$, the path $P_2$, the edge $(w2,w1,w0)$ and the path $P_0$.\\\\

\begin{figure}[h]
\centering
\includegraphics{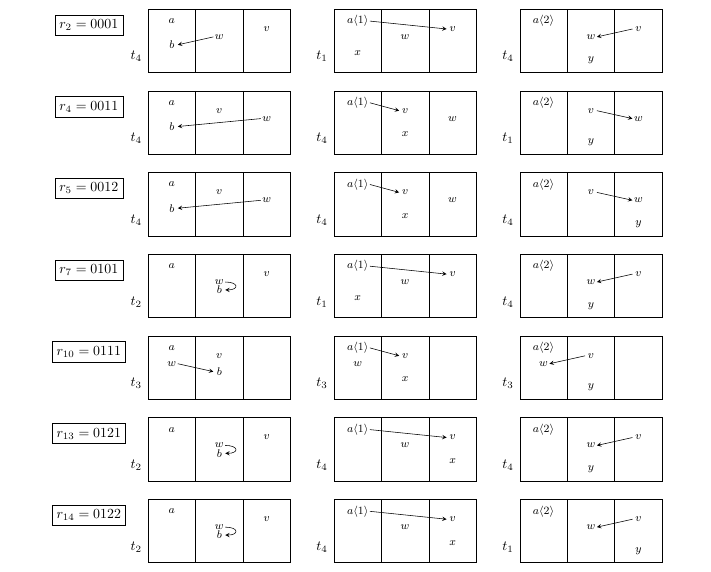}
\captionsetup{width=0.88\textwidth, font=small}
\caption{
The constructions for $\type{2}{d}$ -- the setup is the same as in Figure~\ref{fig_t1}.
}\label{fig_t2}
\end{figure}
\noindent \textbf{Case 3: $\configuration$ is of type $\type{3}{d}$.}
Again, we assume wlog that $\configuration$ is split at the last coordinate,
so the last row is 0101.
We know that $\startvx$ and $\firstskipvx$ are two distinct vertices in the same layer, and analogously to case~1 we can restrict the penultimate row to $r_3,r_4,r_5,r_9,r_{10},r_{11},r_{12},r_{13}$ or $r_{14}$.

We observe that if the penultimate row is $r_5=0012$, then switching the last two coordinates results in a configuration of type $\type{2}{d}$ with penultimate row $r_7$,
which we already handled in case~2.
We also observe that switching $\startvx,\evx$ and also switching $\firstskipvx,\secondskipvx$
also gives a configuration of type $\type{3}{d}$, and with the appropriate permutation of entry values, row $r_{11}$ becomes row $r_{12}$.
Moreover, if the penultimate row is $r_{12}$, then switching $\firstskipvx,\secondskipvx$ and switching the last two coordinates gives a configuration of type~$\type{1}{d}$ with penultimate row~$r_9$, which was already covered in case~1. This means we can exclude rows $r_5, r_{11}$ and $r_{12}$.

Figure~\ref{fig_t3} shows how we construct loose paths $P_0,P_1,P_2$ within the appropriate layers for $r_3,r_4,r_9,r_{10},r_{13}$ or $r_{14}$;
we concatenate the path $P_0$, the edge $(v0,v1,v2)$, the path $P_2$, the edge $(w2,w0,w1)$ and the path $P_1$.
Note that in this case, the path $P_2$ is a LHP in $L_2$, i.e.\ with no omitted vertices in this layer.\end{proof}

\begin{figure}[h]
\centering
\includegraphics{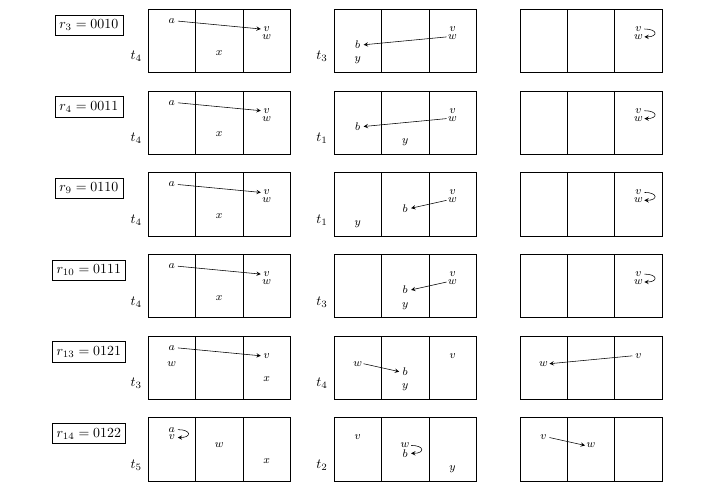}
\captionsetup{width=0.88\textwidth, font=small}
\caption{
The constructions for $\type{3}{d}$. Again, the setup is the same as in Figure~\ref{fig_t1}.
}\label{fig_t3}
\end{figure}

The proof of Proposition~\ref{lem d-1aufd} 
will also be used to prove Lemma~\ref{d-1aufd}:  we will take this proof and show that it is in fact sufficient to assume $\ivsec{d-1}$ instead of $\ivfirst{d-1}$. To do so we introduce the following (slightly technical-looking but ultimately elementary) claim.

\begin{claim} \label{auxiliary coordinates}
            Let  $d' \geq 3$ and let $\alpha,\beta,\gamma,\delta$ be any four distinct vertices in $V_{d'}$.  
            \begin{enumerate}[(i)]
                \item \label{i} There exist distinct vertices $\firsthelper',\secondhelper'\in V_{d'}\setminus\{\alpha,\beta,\gamma,\delta\}$ and $i,j \in [d']$ such that $\firsthelper'_i=\alpha_i$ and $\firsthelper'_j=\beta_j$.
                \item \label{ii} There exist distinct vertices $\firsthelper',\secondhelper'\in V_{d'}\setminus\{\alpha,\beta,\gamma,\delta\}$, there exists $j \in [d']$ and there exists $\mathcal{I} \in \binom{[d']}{2}$ such that $\firsthelper'_j=\alpha_j$ and $\firsthelper'_i=\secondhelper'_i$ for each $i \in \mathcal{I}$.
            \end{enumerate}
        \end{claim}

Informally, we may restate the two properties as follows:
\begin{enumerate}[(i)]
    \item $\firsthelper'$ agrees with each of $\alpha$ and $\beta$ on at least one coordinate.
    \item $\firsthelper'$ agrees with $\alpha$ on at least one coordinate and $\firsthelper'$ agrees with $\secondhelper'$ on at least two coordinates. 
\end{enumerate}
Note that~\eqref{i} makes no claims about $\secondhelper'$ except that it is distinct from all of $\alpha,\beta,\gamma,\delta,\firsthelper'$.
We will apply Claim~\ref{auxiliary coordinates} in a situation where we have fixed the last two coordinates of auxiliary vertices $\firsthelper$ and $\secondhelper$. Setting $d'=d-2$, the claim will enable us to choose the first $d-2$ coordinates, represented by $\firsthelper',\secondhelper'$, in a helpful way.

        \begin{proof}[Proof of Claim~\ref{auxiliary coordinates}]
            \begin{enumerate}[(i)]
                \item 
            Suppose first that there exists an $i \in [d']$ such that $\alpha'_i=\beta'_i$, wlog $i=1$. We choose $\firsthelper'_{1}=\alpha_{1}$, choose $\firsthelper'_2 \notin \{\alpha_2,\beta_2\}$ and $\firsthelper'_3 \notin \{\gamma_3,\delta_3\}$,
            so we have $\firsthelper \notin \{\alpha,\beta,\gamma,\delta\}$.
            For $\secondhelper'$ we choose any arbitrary vertex in $V_{d'} \setminus \{\alpha,\beta,\gamma,\delta,\firsthelper'\}$ -- we can do this since $|V_{d'}|=3^{d'} \geq 3^3 =27 > 5$.
            
            On the other hand, suppose there does not exist $i \in [d']$ such that $\alpha_i=\beta_i$. Then we choose $\firsthelper'_1=\alpha_1$, $\firsthelper'_2=\beta_2$ and $\firsthelper'_3 \notin \{\gamma_3,\delta_3\}$. Since we assumed that $\alpha$ and $\beta$ differ in every coordinate, the vertex $\firsthelper'$ is by definition not in $\{\alpha,\beta\}$. Again we choose $\secondhelper'$ to be any arbitrary vertex in $V_{d'} \setminus \{\alpha,\beta,\gamma,\delta,\firsthelper'\}$. 
           
            \item  We set $\firsthelper'_1=\alpha_1$, set $\secondhelper'_1 \neq \firsthelper'_1$, choose $\firsthelper'_2=\secondhelper'_2 \notin \{\alpha_2,\beta_2\}$ and $\firsthelper'_3=\secondhelper'_3 \notin \{\gamma_3,\delta_3\}$. It is elementary to verify that all the necessary conditions are satisfied.\qedhere
             \end{enumerate}
        \end{proof}

    \begin{proof}[Proof of Lemma~\ref{d-1aufd}]
        We adapt the proof of Proposition~$\ref{lem d-1aufd}$ to show that it is sufficient to assume $\ivsec{{d-1}}$ to prove $\ivfirst{d}$. Again we iterate through all possibilities
        of the type of the configuration, splitting at the last coordinate, and through all possibilities for the penultimate row.
        Then we choose our auxiliary vertices using Claim~$\ref{auxiliary coordinates}$ such that every configuration from $\menge{d-1}$ from the proof of Proposition~$\ref{lem d-1aufd}$ which must be covered by a loose path is in fact a configuration from $\mengeeingeschrankt{d-1}$. This means we have to choose the auxiliary vertices $\firsthelper$ and $\secondhelper$ such that every configuration used is of one of the types $\restriction{1}, \restriction{2}, \restriction{3}$ or $\restriction{4}$. We already determined the $(d-1)$st coordinate in the proof of Proposition~$\ref{lem d-1aufd}$, but we still have the freedom to choose the first $d-2$ coordinates.

In what follows below, for simplicity we abuse notation by writing, for example, $\startvx$ instead of $\startvx_{[d-1]}$. Note in particular that $\startvx\abr{s}_{[d-1]} = \startvx_{[d-1]}$ for any $s$.

        \noindent \textbf{Case 1: $\configuration$ is of type $\type{1}{d}$.}
        As before, we assume that the configuration $\configuration$ is of type  $\type{1}{d}$ split at the last coordinate and individually consider all the possibilities for the penultimate row $r_3,r_4,r_5,r_9,r_{10},r_{11},r_{12},r_{13}$ and $r_{14}$:
        \begin{enumerate}
            \item[$r_3$:]
        We used the $\type{1}{d-1}$-configurations $\firsthelper\secondhelper\firstskipvx\startvx$, $\startvx\firsthelper\secondskipvx\secondhelper$ and the $\type{4}{d-1}$-configuration $\secondhelper\evx\startvx\firsthelper$. For a $\type{1}{d-1}$-configuration to be of type $\restriction{1}$ the ending vertex (i.e.\ the second vertex listed in the configuration)
        must agree in at least one coordinate with at least one of the other vertices and for a $\type{4}{d-1}$-configuration to be of type $\restriction{4}$ the starting vertex has to agree in at least one coordinate with at least one of the other vertices. We therefore need $\secondhelper$ to agree with one of $\firsthelper,\firstskipvx,\startvx$ in a coordinate,
        $\firsthelper$ to agree with one of $\startvx,\secondskipvx,\secondhelper$ in a coordinate,
        and $\secondhelper$ to agree with one of $\evx,\startvx,\firsthelper$ in a coordinate.    
        Claim~\ref{auxiliary coordinates}~\eqref{ii} tells us that we can find $\firsthelper$ and $\secondhelper$ such that they agree with each other in two of the first $d-2$ coordinates, which already ensures that all three necessary conditions are fulfilled.
                \end{enumerate}

        We continue in this manner for the remaining possibilities for the penultimate row. For clarity we present our information in the following table:
        The first column gives the penultimate row, the second shows the configurations we use within each of the three layers, the third the types of these configurations.
        In the next three columns, we observe that to be of the appropriate type $\restriction{i}$,
        we need
        some vertices to agree with others in a certain number of coordinates. A comma in these columns should be understood as ``or'', so for example, if the $u_1,u_2$ appears in the fourth column and $u_3,u_4$ in the fifth, this means that one of $u_1,u_2$ must agree with one of $u_3,u_4$ in at least $n$ coordinates, where $n$ is the entry in the sixth column.
        Alternatively, if the entries are $u_1/u_2$ and $u_3/u_4$ in columns four and five, this means that either $u_1$ has to agree with $u_3$ in $n$ coordinates or $u_2$ has to agree with $u_4$ in $n$ coordinates. 
        Finally, the last column shows how we choose (the first $d-2$ coordinates of) the auxiliary vertices $\firsthelper,\secondhelper$ to agree with other vertices such that all necessary conditions are met. For example, with penultimate row $r_4$ we have two distinct entries in this final column, namely $(\firsthelper,\secondhelper)$ and $(\firsthelper,\startvx)$. This indicates that we choose $\firsthelper$ to agree with $\startvx$ in at least one coordinate and with $\secondhelper$ in at least one coordinate.

        \vspace{0.5cm}

        \begin{tabular}{|c|c|c||c|c|c||c|}
        \hline 
            Row & config & type & vx & agrees with & $n$ & pairs \\
        \hline
             $r_4$& $\firsthelper\secondhelper\firstskipvx\startvx$ & $\shorttype{1}$ & $\secondhelper$ & $\firsthelper,\firstskipvx,\startvx$ & $1$ & $(\firsthelper,\secondhelper)$ \\
             & $\secondhelper\evx\startvx\firsthelper$ & $\shorttype{4}$ & $\secondhelper$ & $\evx,\startvx,\firsthelper$ & $1$ & $(\firsthelper,\secondhelper)$\\
             & $\startvx\firsthelper\secondskipvx\secondhelper$ & $\shorttype{4}$ & $\startvx$ & $\firsthelper,\secondskipvx,\secondhelper$ & $1$ & $(\firsthelper,\startvx)$\\
        \hline
            $r_5$ & $\firsthelper\secondhelper\firstskipvx\startvx$ & $\shorttype{4}$ & $\firsthelper$ & $\secondhelper,\firstskipvx,\startvx$ & $1$ & $(\firsthelper,\secondhelper)$ \\
                  & $\secondhelper\evx\startvx\firsthelper$ & $\shorttype{4}$ & $\secondhelper$ & $\evx,\startvx,\firsthelper$ & $1$ & $(\firsthelper,\secondhelper)$ \\
                  & $\startvx\firsthelper\secondskipvx\secondhelper$ & $\shorttype{4}$ & $\startvx$ & $\firsthelper,\secondskipvx,\secondhelper$ & $1$ & $(\firsthelper,\startvx)$ \\
        \hline
            $r_9$ & $\firsthelper\secondhelper\firstskipvx\startvx$ & $\shorttype{4}$ & $\firsthelper$ & $\secondhelper, \firstskipvx,\startvx$ & $1$ & $(\firsthelper,\secondhelper)$ \\
                  & $\secondhelper\evx\startvx\firsthelper$ & $\shorttype{2}$ & $\secondhelper,\evx$ & $\startvx,\firsthelper$ & $2$ & $(\firsthelper,\secondhelper)$ \\
                  & $\startvx\firsthelper\secondskipvx\secondhelper$ & $\shorttype{1}$ & $\firsthelper$ & $\startvx,\secondskipvx,\secondhelper$ & $1$ & $(\firsthelper,\secondhelper)$\\
        \hline
            $r_{10}$ & $\firsthelper\secondhelper\firstskipvx\startvx$ & $\shorttype{3}$ & $\firsthelper/\secondhelper$ & $\startvx/\firstskipvx$ & $1$ & $(\firsthelper,\startvx)$\\
                     & $\secondhelper\evx\startvx\firsthelper$ & $\shorttype{3}$& $\secondhelper/\evx$ & $\firsthelper/\startvx$ & $1$ & $(\firsthelper,\secondhelper)$\\
                     & $\startvx\firsthelper\secondhelper\secondskipvx$ & $\shorttype{3}$ & $\startvx/\firsthelper$ & $\secondskipvx/\secondhelper$ & $1$ & $(\firsthelper,\secondhelper)$\\
        \hline
            $r_{11}$ & $\firsthelper\secondhelper\firstskipvx\startvx$ & $\shorttype{4}$ & $\firsthelper$ & $           \secondhelper,\firstskipvx,\startvx$ & $1$ & $(\firsthelper,\secondhelper)$ \\
                     & $\secondhelper\evx\startvx\firsthelper$ & $\shorttype{2}$ & $\secondhelper,\evx$ & $\startvx,\firsthelper$ & $2$ & $(\firsthelper,\secondhelper)$ \\
                     & $\startvx\firsthelper\secondskipvx\secondhelper$ & $\shorttype{4}$ & $\startvx$ & $\firsthelper,\secondskipvx,\secondhelper$ & $1$ & $(\firsthelper,\startvx)$ \\
        \hline
            $r_{12}$ & $\firsthelper\secondhelper\firstskipvx\startvx$ & $\shorttype{1}$ & $\secondhelper$ & $\firsthelper,\firstskipvx,\startvx$ & $1$ & $(\firsthelper,\secondhelper)$ \\
                     & $\secondhelper\evx\startvx\firsthelper$ & $\shorttype{2}$ & $\secondhelper,\evx$ & $\startvx,\firsthelper$ & $2$ & $(\firsthelper,\secondhelper)$ \\
                     & $\startvx\firsthelper\secondhelper\secondskipvx$ & $\shorttype{3}$ & $\startvx/\firsthelper$ & $\secondskipvx/\secondhelper$ & $1$ & $(\firsthelper,\secondhelper)$\\
        \hline
            $r_{13}$ & $\firsthelper\secondhelper\startvx\firstskipvx$ & $\shorttype{4}$ & $\firsthelper$ & $         \secondhelper,\startvx,\firstskipvx$ & $1$ & $(\firsthelper,\secondhelper)$ \\
                     & $\secondhelper\evx\startvx\firsthelper$ & $\shorttype{3}$& $\secondhelper/\evx$ & $\firsthelper/\startvx$ & $1$ & $(\firsthelper,\secondhelper)$\\
                     & $\startvx\firsthelper\secondhelper\secondskipvx$ & $\shorttype{3}$ & $\startvx/\firsthelper$ & $\secondskipvx/\secondhelper$ & $1$ & $(\firsthelper,\secondhelper)$\\
        \hline
            $r_{14}$ & $\firsthelper\secondhelper\firstskipvx\startvx$ & $\shorttype{1}$ & $\secondhelper$ &            $\firsthelper,\firstskipvx,\startvx$ & $1$ & $(\firsthelper,\secondhelper)$ \\
                     & $\secondhelper\evx\startvx\firsthelper$ & $\shorttype{2}$ & $\secondhelper,\evx$ & $\startvx,\firsthelper$ & $2$ & $(\firsthelper,\secondhelper)$ \\
                     & $\startvx\firsthelper\secondskipvx\secondhelper$ & $\shorttype{4}$ & $\startvx$ & $\firsthelper,\secondskipvx,\secondhelper$ & $1$ & $(\firsthelper,\startvx)$ \\
        \hline
        \end{tabular}

\vspace{0.5cm}
        
        \noindent \textbf{Case 2: $\configuration$ is of type $\type{2}{d}$.}
        We follow the same principle as in case 1 -- the relevant table is shown below.

\vspace{0.5cm}
        
        \begin{tabular}{|c|c|c||c|c|c||c|}
        \hline 
            Row & config & type & vx & agrees with & $n$ & pairs \\
            \hline
             $r_2$& $\secondhelper\evx\startvx\firsthelper$ & $\shorttype{4}$ & $\secondhelper$ & $\evx,\startvx,\firsthelper$ & $1$ & $(\firsthelper,\secondhelper)$ \\
             & $\startvx\firsthelper\firstskipvx\secondhelper$ & $\shorttype{1}$ & $\firsthelper$ & $\startvx,\firstskipvx,\secondhelper$ & $1$ & $(\firsthelper,\secondhelper)$\\
             & $\firsthelper\secondhelper\secondskipvx\startvx$ & $\shorttype{4}$ & $\firsthelper$ & $\secondhelper,\secondskipvx,\startvx$ & $1$ & $(\firsthelper,\secondhelper)$\\
             \hline
             $r_4$& $\secondhelper\evx\startvx\firsthelper$ & $\shorttype{4}$ & $\secondhelper$ & $\evx,\startvx,\firsthelper$ & $1$ & $(\firsthelper,\secondhelper)$ \\
             & $\startvx\firsthelper\firstskipvx\secondhelper$ & $\shorttype{4}$ & $\startvx$ & $\firsthelper,\firstskipvx,\secondhelper$ & $1$ & $(\firsthelper,\startvx)$\\
             & $\firsthelper\secondhelper\secondskipvx\startvx$ & $\shorttype{1}$ & $\secondhelper$ & $\firsthelper,\secondskipvx,\startvx$ & $1$ & $(\firsthelper,\secondhelper)$\\
             \hline
             $r_5$& $\secondhelper\evx\startvx\firsthelper$ & $\shorttype{4}$ & $\secondhelper$ & $\evx,\startvx,\firsthelper$ & $1$ & $(\firsthelper,\secondhelper)$ \\
             & $\startvx\firsthelper\firstskipvx\secondhelper$ & $\shorttype{4}$ & $\startvx$ & $\firsthelper,\firstskipvx,\secondhelper$ & $1$ & $(\firsthelper,\startvx)$\\
             & $\firsthelper\secondhelper\secondskipvx\startvx$ & $\shorttype{4}$ & $\firsthelper$ & $\secondhelper,\secondskipvx,\startvx$ & $1$ & $(\firsthelper,\secondhelper)$\\
             \hline
             $r_7$& $\secondhelper\evx\startvx\firsthelper$ & $\shorttype{2}$ & $\secondhelper,\evx$ & $\startvx,\firsthelper$ & $2$ & $(\firsthelper,\secondhelper)$ \\
             & $\startvx\firsthelper\firstskipvx\secondhelper$ & $\shorttype{1}$ & $\firsthelper$ & $\startvx,\firstskipvx,\secondhelper$ & $1$ & $(\firsthelper,\secondhelper)$\\
             & $\firsthelper\secondhelper\secondskipvx\startvx$ & $\shorttype{4}$ & $\firsthelper$ & $\secondhelper,\secondskipvx,\startvx$ & $1$ & $(\firsthelper,\secondhelper)$\\
             \hline
             $r_{10}$& $\secondhelper\evx\startvx\firsthelper$ & $\shorttype{3}$ & $\secondhelper/\evx$ & $\firsthelper/\startvx$ & $1$ & $(\firsthelper,\secondhelper)$ \\
             & $\startvx\firsthelper\secondhelper\firstskipvx$ & $\shorttype{3}$ & $\startvx/\firsthelper$ & $\firstskipvx/\secondhelper$ & $1$ & $(\firsthelper,\secondhelper)$\\
             & $\firsthelper\secondhelper\secondskipvx\startvx$ & $\shorttype{3}$ & $\firsthelper,\secondhelper$ & $\secondskipvx,\startvx$ & $1$ & $(\firsthelper,\secondskipvx)$\\
             \hline
             $r_{13}$& $\secondhelper\evx\startvx\firsthelper$ & $\shorttype{2}$ & $\secondhelper,\evx$ & $\startvx,\firsthelper$ & $2$ & $(\firsthelper,\secondhelper)$ \\
            &  $\startvx\firsthelper\firstskipvx\secondhelper$ & $\shorttype{4}$ & $\startvx$ & $\firsthelper,\firstskipvx,\secondhelper$ & $1$ & $(\firsthelper,\startvx)$\\
             & $\firsthelper\secondhelper\secondskipvx\startvx$ & $\shorttype{4}$ & $\firsthelper$ & $\secondhelper,\secondskipvx,\startvx$ & $1$ & $(\firsthelper,\secondhelper)$\\
             \hline
             $r_{14}$& $\secondhelper\evx\startvx\firsthelper$ & $\shorttype{2}$ & $\secondhelper,\evx$ & $\startvx,\firsthelper$ & $2$ & $(\firsthelper,\secondhelper)$ \\
             & $\startvx\firsthelper\firstskipvx\secondhelper$ & $\shorttype{4}$ & $\startvx$ & $\firsthelper,\firstskipvx,\secondhelper$ & $1$ & $(\firsthelper,\startvx)$\\
             & $\firsthelper\secondhelper\secondskipvx\startvx$ & $\shorttype{1}$ & $\secondhelper$ & $\firsthelper,\secondskipvx,\startvx$ & $1$ & $(\firsthelper,\secondhelper)$\\
        \hline
        \end{tabular}

        \vspace{0.5cm}
        
        \noindent \textbf{Case 3: $\configuration$ is of type $\type{3}{d}$.}
        Case 3 is slightly different from the previous two cases in that one layer is covered by a loose Hamilton path, so we do not have a configuration type to worry about in this layer. 
        Thus for each possible penultimate row, we have only two configuration types in the table below. We now have to iterate through the possibilities for the penultimate row $r_3,r_4,r_9,r_{10},r_{13}$ or $r_{14}$.
        
        We observe that for $r_9$ we used exactly the same types of configuration within layers as for $r_4$, with only permutations of the values of coordinate $d-1$ (though different permutations for each layer). We can thus argue for $r_9$ exactly as for $r_4$. Similarly, the argument for $r_{10}$ is identical to that for $r_3$.
        
        This leaves us with the rows $r_3,r_4,r_{13},r_{14}$. Rows $r_3,r_4$ are the only rows from any of the three cases in this proof for which we cannot use Claim~\ref{auxiliary coordinates}~\eqref{ii}. For both rows, we apply Claim~\ref{auxiliary coordinates}~\eqref{i} with $(\alpha,\beta,\gamma,\delta):=(\startvx_{[d-2]},\evx_{[d-2]},\firstskipvx_{[d-2]},\secondskipvx_{[d-2]})$, obtain $\firsthelper',\secondhelper'$ and set $\firsthelper{[d-2]}:= \firsthelper'$ and $\secondhelper{[d-2]}:= \secondhelper'$.
        Now by the statement of the claim, $\firsthelper$ agrees with each of $\startvx$ and $\evx$ in one of the first $d-2$ coordinates.
        This ensures that our configurations are of types $\restriction{1},\restriction{3}$ or $\restriction{4}$ as appropriate.

\vspace{0.5cm}

        \begin{tabular}{|c|c|c||c|c|c||c|}
        \hline 
            Row & config & type & vx & agrees with & $n$ & pairs \\
        \hline
            $r_3$& $\startvx\firsthelper\secondhelper\firstskipvx$ & $\shorttype{4}$ & $\startvx$ & $\firsthelper,\secondhelper,\firstskipvx$ & $1$ & $(\firsthelper,\startvx)$ \\
             & $\secondhelper\evx\firsthelper\secondskipvx$ & $\shorttype{3}$ & $\secondhelper/\evx$ & $\secondskipvx/\firsthelper$ & $1$ & $(\firsthelper,\evx)$\\
             \hline 
             $r_4$& $\startvx\firsthelper\secondhelper\firstskipvx$ & $\shorttype{4}$ & $\startvx$ & $\firsthelper,\secondhelper,\firstskipvx$ & $1$ & $(\firsthelper,\startvx)$ \\
             & $\secondhelper\evx\firsthelper\secondskipvx$ & $\shorttype{1}$ & $\evx$ & $\secondhelper,\firsthelper,\secondskipvx$ & $1$ & $(\firsthelper,\evx)$\\
             \hline 
             $r_{13}$& $\startvx\firsthelper\secondhelper\firstskipvx$ & $\shorttype{3}$ & $\startvx/\firsthelper$ & $\firstskipvx/\secondhelper$ & $1$ & $(\firsthelper,\secondhelper)$ \\
             & $\secondhelper\evx\secondskipvx\firsthelper$ & $\shorttype{4}$ & $\secondhelper$ & $\evx,\secondskipvx,\firsthelper$ & $1$ & $(\firsthelper,\secondhelper)$\\
             \hline 
             $r_{14}$& $\startvx\firsthelper\secondhelper\firstskipvx$ & $\shorttype{5}$ & $\startvx,\firsthelper$ & $\secondhelper,\firstskipvx$ & $2$ &  $(\firsthelper,\secondhelper)$\\
             & $\secondhelper\evx\firsthelper\secondskipvx$ & $\shorttype{2}$ & $\secondhelper,\evx$ & $\firsthelper, \secondskipvx$ & $2$ & $(\firsthelper,\secondhelper)$\\
        \hline
        \end{tabular}

        \vspace{0.5cm}

This completes the proof of Lemma~\ref{d-1aufd}.
    \end{proof}

\section{Concluding remarks}\label{sec: concluding remarks}

The themes addressed in this paper raise a number of open questions for further research.
In particular, one could ask about the existence of loose Hamilton paths in cube hypergraphs of higher uniformities, i.e.\ in $\kuniformcube{d}$ for $k \ge 4$. Note that $\kuniformcube{d}$ contains $k^d$ vertices while a loose path of length $\ell:= k^{d-1}+k^{d-2}+\ldots+1$ contains $1+\ell(k-1) = k^d$ vertices, so the natural necessary divisibility condition for the existence of a loose Hamilton path is certainly satisfied.

Similar considerations to the case $k=3$ readily show that, when $k\ge 4$, there is a loose Hamilton path (trivially) for $d=1$ but no such path for $d=2,3$. It seems likely that the larger $k$ becomes, the larger the dimension $d$ required before a loose Hamilton path exists.

\begin{quest}\ 
\begin{enumerate}
    \item Given $k \in \mathbb{N}$, does there exist a $d_0=d_0(k)\ge 2$ such that $\kuniformcube{d_0}$ contains a loose Hamilton path?
    \item Does $d_0$ tend to infinity with $k$?
    \item Is it even true that $\kuniformcube{d}$ contains a loose Hamilton path for every $d\ge d_0$?
\end{enumerate}
\end{quest}

It seems plausible that the answer to all of these questions is yes, but proving this would require new ideas.

In the introduction we mentioned several papers which investigated the conjecture of
Ruskey and Savage~\cite{RS93}
that in the (graph) hypercube $Q_d$, any matching can be extended to a Hamilton cycle, and which proved partial results in this direction~\cite{AAAHST15,Fink07,Gregor09}.
The analogous statement for loose Hamilton paths in cube hypergraphs is easily seen to be false, even for dimensions in which these exist: cube hypergraphs contain perfect matchings, but loose Hamilton paths do not. However, one could place restrictions either on the matching (its size and structure) or on the number of edges of the matching which must be used. A $3/4$ proportion is an easy asymptotic (in $d$) upper bound on the proportion which can plausibly be guaranteed to be used in $\cube{d}$, since a perfect matching contains $3^{d-1}$ edges and the largest matching in a LHP contains $\lceil\frac{3^d-1}{4}\rceil$ edges.

One can also ask further questions that have also been addressed for graphs, for instance about the number of loose Hamilton paths, or their robustness against edge deletion.

One can also consider other possible generalisations of paths and cycles to hypergraphs, as was done in~\cite{JMthesis}. Here it was determined for exactly which dimensions $d$ the cube hypergraph $\cube{d}$ contains a Berge Hamilton cycle, or a $2$-vertex-connected $2$-factor. The latter of these two results also generalises naturally to $\kuniformcube{d}$, but in general it is still an open question whether Berge Hamilton cycles exist in $\kuniformcube{d}$.

\newpage
\bibliographystyle{plain}
\bibliography{main.bib}

\begin{thebibliography}{10}

\bibitem{AAAHST15}
Adel Alahmadi, Robert E.~L. Aldred, Ahmad Alkenani, Rola Hijazi, P.~Sol\'{e}, and Carsten Thomassen.
\newblock Extending a perfect matching to a {H}amiltonian cycle.
\newblock {\em Discrete Math. Theor. Comput. Sci.}, 17(1):241--254, 2015.

\bibitem{BC95}
Gustav Burosch and Pier~Vittorio Ceccherini.
\newblock Isometric embeddings into cube-hypergraphs.
\newblock {\em Discrete Math.}, 137(1-3):77--85, 1995.

\bibitem{BC96}
Gustav Burosch and Pier~Vittorio Ceccherini.
\newblock A characterization of cube-hypergraphs.
\newblock {\em Discrete Math.}, 152(1-3):55--68, 1996.

\bibitem{CL91}
Mee~Yee Chan and Shiang-Jen Lee.
\newblock On the existence of {H}amiltonian circuits in faulty hypercubes.
\newblock {\em SIAM J. Discrete Math.}, 4(4):511--527, 1991.

\bibitem{Clark97}
Lane Clark.
\newblock A new upper bound for the number of {H}amiltonian cycles in the {$n$}-cube.
\newblock volume~25, pages 35--37. 2000.
\newblock Recent advances in interdisciplinary mathematics (Portland, ME, 1997).

\bibitem{CEGKO21}
Padraig Condon, Alberto Espuny~D\'{\i}az, Ant\'{o}nio Gir\~{a}o, Daniela K\"{u}hn, and Deryk Osthus.
\newblock Hamiltonicity of random subgraphs of the hypercube.
\newblock In {\em Proceedings of the 2021 {ACM}-{SIAM} {S}ymposium on {D}iscrete {A}lgorithms ({SODA})}, pages 889--898. [Society for Industrial and Applied Mathematics (SIAM)], Philadelphia, PA, 2021.

\bibitem{code}
Oliver Cooley, Johannes Machata, and Matija Pasch.
\newblock Source code for the computation of the missing normalised configurations in the base case, 2024.
\newblock \url{https://github.com/MatijaPasch/CubeLHP}.

\bibitem{DG75}
E.~Dixon and S.~Goodman.
\newblock On the number of {H}amiltonian circuits in the {$n$}-cube.
\newblock {\em Proc. Amer. Math. Soc.}, 50:500--504, 1975.

\bibitem{Douglas77}
Robert~James Douglas.
\newblock Bounds on the number of {H}amiltonian circuits in the {$n$}-cube.
\newblock {\em Discrete Math.}, 17(2):143--146, 1977.

\bibitem{DV21}
Pavel Dvo\v{r}\'{a}k and Tom\'{a}\v{s} Valla.
\newblock Automorphisms of the cube {$n^d$}.
\newblock {\em Discrete Math.}, 344(3):Paper No. 112234, 14, 2021.

\bibitem{FS09}
Tom\'{a}s Feder and Carlos Subi.
\newblock Nearly tight bounds on the number of {H}amiltonian circuits of the hypercube and generalizations.
\newblock {\em Inform. Process. Lett.}, 109(5):267--272, 2009.

\bibitem{Fink07}
Ji\v{r}\'{\i} Fink.
\newblock Perfect matchings extend to {H}amilton cycles in hypercubes.
\newblock {\em J. Combin. Theory Ser. B}, 97(6):1074--1076, 2007.

\bibitem{Gregor09}
Petr Gregor.
\newblock Perfect matchings extending on subcubes to {H}amiltonian cycles of hypercubes.
\newblock {\em Discrete Math.}, 309(6):1711--1713, 2009.

\bibitem{JMthesis}
Johannes Machata.
\newblock Hamilton paths and cycles in the 3-uniform cube hypergraph.
\newblock {B}achelor's {T}hesis, LMU Munich, 2024.
\newblock \href{https://arxiv.org/abs/2406.00401v1}{arXiv: 2406.00401v1}.

\bibitem{Mollard88}
M.~Mollard.
\newblock Un nouvel encadrement du nombre de cycles hamiltoniens du {$n$}-cube.
\newblock {\em European J. Combin.}, 9(1):49--52, 1988.

\bibitem{RS93}
Frank Ruskey and Carla Savage.
\newblock Hamilton cycles that extend transposition matchings in {C}ayley graphs of {$S_n$}.
\newblock {\em SIAM J. Discrete Math.}, 6(1):152--166, 1993.

\bibitem{WJT22}
Na~Wang, Jixiang Meng, and Yingzhi Tian.
\newblock Connectivity of {C}artesian product of hypergraphs.
\newblock {\em Bull. Iranian Math. Soc.}, 48(5):2379--2393, 2022.

\end{thebibliography}
\appendix
\section{Pseudo-code} \label{app:code}

In this appendix we describe the computer code used to check that all normalised configurations of four vertices in $\cube{4}$ except the four described in Section~\ref{sec: base case} can be covered by loose paths. The code itself can be found at \cite{code}.

We will describe the algorithms using pseudo-code, in which our notation will follow that of the paper rather than that in the code files.
We denote by $N$ the set of all normalised configurations of four vertices and by \textit{\lpskipped} a loose path containing all but two vertices (so an \lpskipped\ will cover an appropriate configuration of four vertices).

The code is divided into the three programs \allnc, \coveredc\ and \notcoveredc.
\begin{enumerate}[(i)]
    \item \text{\allnc} creates the set \textit{\setallnc} of all normalised configurations of four vertices.
    \item \text{\coveredc} takes as input a set $P$ of (encoded) paths. First, it verifies that these are indeed {\lpskipped}s, then it checks which configurations of four vertices they cover and writes these configurations to a set $\setallcoveredc$. 

    \item \notcoveredc\ generates the set $\setallnotcoveredc:= \setallnc \setminus \setallcoveredc $.
\end{enumerate}

The main difficulty was to create the set $P$, which was done via a computer search for {\lpskipped}s. We do not include a description of this program here since it is not necessary in order to check the correctness of the proof -- the code used to create the set $P$ can be found in  \cite{code}. \\\\

\subsection{\allnc}
Since runtime complexity is not an issue, our primary focus was transparency rather than efficiency. First, we start with the set $\setallnc$ being empty. Then, we iterate through all possible $(\startvx,\evx,\firstskipvx,\secondskipvx)$ in $\{0,1,2\}^4$. According to Section \ref{sec:normalisation}, we discard the configuration if the vertices are not distinct. Then we normalise the configuration $(\startvx,\evx,\firstskipvx,\secondskipvx)$ and add it to $\setallnc$ if it is not already present. Thus when the program terminates, the set $\setallnc$ is the set of all normalised configurations. 

\begin{algorithmic}
\State $\setallnc := \emptyset$
\State $\tilde{\setallnc}:=\{0,1,2\}^4$
\ForAll{$\startvx,\evx,\firstskipvx,\secondskipvx$ {in} $\Tilde{N}$}
\If{$\startvx,\evx,\firstskipvx,\secondskipvx$ {are distinct}}
\State $c := (\startvx,\evx,\firstskipvx,\secondskipvx)$ normalised \Comment{The normalisation algorithm works analogously to Section~\ref{sec:normalisation}}
\If{$c$ {is not already in} $N$}
\State add $c$ to $\setallnc$
\EndIf
\EndIf
\EndFor
\end{algorithmic}

\subsection{\coveredc}
We import the set $P$, which is a set of tuples, with each tuple consisting of 81 vertices. The first 79 vertices represent the vertices of a loose path, while the last two vertices are two additional (omitted) vertices.
We introduce the empty set $\setallcoveredc$ and then iterate through all {\lpskipped}s in $P$. If a path is an {\lpskipped} (it always is, for our choice of $P$), we add its configuration to $\setallcoveredc$. Thus, in the end $\setallcoveredc$ only contains covered configurations.

\begin{algorithmic}
\State Import $P$
\State $\setallcoveredc = \emptyset$
\ForAll{$p$ in $P$}
\If{$p$ is an \lpskipped}
\State  Add the configuration $c$ of $p$ to $\setallcoveredc$
\EndIf
\EndFor
\end{algorithmic}
\subsection{\notcoveredc}
We introduce $\setallnotcoveredc$ as the empty set and then iterate through all configurations in $\setallnc$. If a configuration is not in $\setallcoveredc$, we add it to $\setallnotcoveredc$. Thus, in the end $\setallnotcoveredc$ only contains configurations, for which we do not know an {\lpskipped}. 

\begin{algorithmic}
\State $\setallnotcoveredc= \emptyset$
\ForAll{$c$ {in} $\setallnc$}
\If{$c$ {is not in} $\setallcoveredc$}
\State add $c$ to $\setallnotcoveredc$
\EndIf
\EndFor
\end{algorithmic}
\end{document}